\UseRawInputEncoding
\documentclass[12pt, reqno]{amsart}
\usepackage{amssymb}
\usepackage{amsfonts}
\usepackage{amssymb,amsmath,amscd}
\usepackage[colorlinks=true]{hyperref}
\usepackage{mathrsfs}
\newtheorem{Theorem}{\indent Theorem}[section]

\newtheorem{Lemma}[Theorem]{\indent Lemma}
\newtheorem{Corollary}[Theorem]{\indent Corollary}

\theoremstyle{remark}
\newtheorem{Remark}{Remark}

\begin{document}
\centerline{
\bf On primes in special sequences with applications to Carmichael numbers
}
\bigskip
\centerline{\small  Wei Zhang}
\bigskip

\textbf{Abstract}
By involving  some exponential sums related to $\Lambda(n)$ in arithmetic progression, we can  obtain some new results for von Mangoldt function over {\bf nonhomogeneous}  Beatty sequences in arithmetic progressions, which  improve some recent results of Banks-Yeager unconditionally. On the other hand, we also considered the primes over Piatetski-Shapiro sequences in arithmetic progressions, which gives a  continuous improvement of the results of \cite{BBB}.  These results can be used to improve some results related to the Carmichael numbers.
\medskip

\medskip


\bigskip
\bigskip
\numberwithin{equation}{section}

\section{Introduction}

In this paper,  we are interested in
the von Mangoldt function over Beatty sequences in arithmetic sequences. The so-called Beatty sequence of integers defined by
\[
\mathcal{B}_{\alpha,\beta}:=\{[\alpha n+\beta]\}_{n=1}^{\infty},
\]
where $\alpha$ and $\beta$ are fixed real numbers and $[x]$ denotes the greatest integer not larger than $x$.
Now we will recall some notion related to the type of $\alpha.$ The definition of an irrational number of constant type can be cited as follows.
For an irrational number $\alpha,$ we define its type $\tau$ by the relation
\[
\tau:=\sup\left\{\theta\in\mathbb{R}:\
\liminf_{\substack{q\rightarrow \infty\\ q\in\mathbb{Z}^{+}}}q^{\theta}\parallel
\alpha q\parallel=0\right\}.
\]
Let $\psi$ be a non-decreasing positive function that defined for integers. The irrational number $\alpha$ is said to be of type $<\psi$ if $q\parallel q\alpha\parallel\geq 1/\psi(q)$ holds for every positive integers $q.$ If $\psi$ is a constant function, then an irrational $\alpha$ is also called a constant type (finite type). This relation between these two definitions is that an individual number $\alpha$ is of type $\tau$ if and only if for every constant $\tau$, there is a constant $c(\tau,\alpha)$ such that $\alpha$ is of type $\tau$ with
$$q\parallel q\alpha\parallel\geq c(\tau,\alpha)q^{-\tau-\varepsilon+1}.$$
The analytic properties of Beatty sequences have been studied by many experts. For example, one can refer to \cite{BS,BY,BS1} and the references therein. Specially, we focus on some such  results appeared in \cite{BS,BY}. Such type results have some relations with a result of Jia \cite{J} and a  conjecture of Long \cite{Long} and draw many authors special attention.
For example, let $\alpha>1$ and $\beta$ be fixed real numbers with $\alpha$ positive, irrational, and of finite type $\tau=\tau(\alpha)$. Then for a constant $\kappa>0$, all integer $c,d$ with $\gcd(c,d)=1,$ in \cite{BS}, it is proved that
\begin{align}\label{BS}
\sum_{n\leq N,\ n\in\mathcal{B}_{\alpha,\beta}\atop
n\equiv c(\textup{mod}\ d)}\Lambda(n)
=\alpha^{-1}\sum_{m\leq N\atop m \equiv c (\textup{mod}\ d)}\Lambda(m)
+O(N^{1-\kappa}).
\end{align}
where the implied constant depends only on $\alpha$ and $\beta.$  Recently, an explicit version of the constant $\kappa$ was given in \cite{BY}. In \cite{BY}, Banks and Yeager showed the follows.
Let $\alpha>1$  be   fixed irrational numbers of finite type $\tau<\infty$ and $\beta$ be fixed real numbers. Then there is a constant $\varepsilon>0$ such that for all integer $c,d$ with $\gcd(c,d)=1,$ one has
\begin{align}\label{BY}
\sum_{n\leq N,\ n\in\mathcal{B}_{\alpha,\beta}\atop
n\equiv c(\textup{mod}\ d)}\Lambda(n)
=\alpha^{-1}\sum_{m\leq N\atop m \equiv c (\textup{mod}\ d)}\Lambda(m)
+O(N^{1-1/(4\tau+2)+\varepsilon}).
\end{align}
where the implied constant depends only on $\alpha,$ $\varepsilon$ and $\beta.$
However, it seems that one can get a much better error term.

\begin{Theorem}\label{th2}
Let $\alpha>1$ be a fixed irrational number of finite type $\tau<\infty$ and $\beta$ be any fixed real number. Then there is a constant $\varepsilon>0$ such that for all integer $c,d$ with $\gcd(c,d)=1,$ we have
\begin{align}\label{Z}
\sum_{n\leq N,\ n\in\mathcal{B}_{\alpha,\beta}\atop
n\equiv c(\textup{mod}\ d)}\Lambda(n)
=\alpha^{-1}\sum_{m\leq N\atop m \equiv c (\textup{mod}\ d)}\Lambda(m)
+O(N^{1-1/(3\tau+2)+\varepsilon}).
\end{align}
where $N$ is a sufficiently large integer and the implied constant depends only on $\alpha,$ $\beta$ and $\varepsilon$.
\end{Theorem}

\begin{Remark}
In \cite{BY}, the authors used a Fourier series of $\psi(x)$ in a restricted interval of $x$ (see Chapter I, Lemma 12 in \cite{Vi}). Hence one can consider such type problems by using sums of arithmetic functions twist exponential functions. This method can be applied to many important arithmetic functions (with possible sign changes) and giving some very strong results. However, by using this method, one may need to deal with one more parameter $\Delta$ to restricted $\psi(x)$ in certain intervals. See equations (9), (11) and line 25 in \cite{BY}. Precisely, one may need to balance the error term
\[
\left(N\Delta+H^{-1}\Delta^{-1}N\right).
\]
When we  use a Fourier expansion of Vaaler \cite{Vaa}, then we don't have to use a restricted domain of $\psi(x).$ Hence, we need to deal the error term
\[
H^{-1}N
\]
and we don't have to deal with the parameter $\Delta$ in \cite{BY}. However, with the Fourier expansion of Vaaler, if we want to deal with such type problems by involving exponential sums, we need the arithmetic functions without possible sign changes. Hence, $\Lambda(n)$ is possible.

In the proof, we need
\[
m\leq n\alpha+\beta<m+1,
\]
to ensure
\begin{align*}
\sum_{n\leq N,\ n\in\mathcal{B}_{\alpha,\beta}\atop
n\equiv c(\textup{mod}\ d)}\Lambda(n)
=\sum_{m\leq N\atop m \equiv c (\textup{mod}\ d)}\Lambda(m)
\left(\left[\frac{m-\beta+1}{\alpha}\right]
-\left[\frac{m-\beta}{\alpha}\right]
\right).
\end{align*}
This means that we need
\begin{align}\label{zzz}
\#\left\{\left[\frac{m-\beta}{\alpha},\frac{m-\beta+1}{\alpha}
\right)\cap\mathbb{Z}\right\}
=\left[\frac{m-\beta+1}{\alpha}\right]
-\left[\frac{m-\beta}{\alpha}\right].
\end{align}
In fact, we only need to ensure that $\frac{n-\beta+1}{\alpha}$ and $\frac{n-\beta}{\alpha}$ are simultaneously integers or not. Hence, one can ensure (\ref{zzz}). Obviously, for fixed $\beta$ of the form $\beta\neq m_{0}+k\alpha,$ $k\in\mathbb{N}$ and $1\leq m_{0}\leq N+1,$ we can ensure (\ref{zzz}). For otherwise, we have an error term $O(\log N).$
\end{Remark}

Theorem \ref{th2} can be used to give related results of Carmichael numbers in Beatty sequences. Carmichael numbers are the composite natural numbers $N$ with the property that $N|(a^{N}-a)$ for every integer $a.$ In 1994, in \cite{AGP}, it is proved that there are infinitely many Carmichael numbers.
Let $\pi(x)$ be the number of primes $p\leq x$, and let $\pi(x,y)$ be the number
of those for which $p-1$ is free of prime factors exceeding $y$. Denote by $\mathcal{E}$ the set of numbers $E$ in the range $0<E<1$ for which there
exist numbers $x_{4}(E),$ $\gamma(E)>0$ such that
$\pi(x,x^{1-E})\geq \gamma(E)\pi(x)$ for all $x>x_{4}(E).$ Let $\mathcal{P}$ denote the set of all prime numbers, and set $\mathcal{P}_{\alpha,\beta}=
\mathcal{P}\cap \mathcal{B}_{\alpha,\beta}.$
For each $E\in\mathcal{E},$ $B\in(0,1/(4\tau+2))$ and $\varepsilon>0,$ in \cite{BY}, it is proved that there is a number $x_{0}$ depending on $\alpha,\beta,\varepsilon,E,B$ that for any $x\geq x_{4}$ there are at least $x^{EB-\varepsilon}$  numbers up to $x$ composed solely of primes from $\mathcal{P}_{\alpha,\beta}.$ Based on our new result Theorem \ref{th2} and the standard arguments on \cite{BY}, we can give the following improved result.
\begin{Corollary} Denote by $\mathcal{E}$ the set of numbers $E$ in the range $0<E<1$ for which there
exist numbers $x_{4}(E),$ $\gamma(E)>0$ such that
$\pi(x,x^{1-E})\geq \gamma(E)\pi(x)$ for all $x>x_{4}(E).$ Let $\mathcal{P}$ denote the set of all prime numbers, and set $\mathcal{P}_{\alpha,\beta}=
\mathcal{P}\cap \mathcal{B}_{\alpha,\beta}.$
For each $E\in\mathcal{E},$ $B\in(0,1/(3\tau+2))$ and $\varepsilon>0,$ there is a number $x_{0}$ depending on $\alpha,\beta,\varepsilon,E,B$ that for any $x\geq x_{4}$ there are at least $x^{EB-\varepsilon}$  numbers up to $x$ composed solely of primes from $\mathcal{P}_{\alpha,\beta}.$
\end{Corollary}

In this paper,  we are also interested in
the primes over Beatty sequences in arithmetic sequences.
Piatetski-Shapiro sequences are named in honor of Piatetski-Shapiro, who proved that for any number $c\in(1,12/11)$ there are infinitely many primes of the form $[n^{c}]$ by showing that
 \begin{align}\label{01}
 \sum_{\substack{1\leq n\leq x\\ [n^{c}]\ \textup{is}\ \textup{prime}}}1
 =(1+o(1))\frac{N}{c\log N}.
 \end{align}
 The admissible range for $c$ in this problem has  been extended by many experts over the years.
And to date, the largest admissible $c$-range for (\ref{01}) seems to be $c\in(2817/2425)$ due to Rivat and Sargos  \cite{RS0} (see also the
references to the previous record holders they gave in their paper). We call such type primes Piatetski-Shapiro primes. Many experts considered the  Piatetski-Shapiro primes in arithmetic progressions. For example, in \cite{LW}, Leitmann and Wolke showed that for $c\in(1,12/11),$ $q\in\mathbb{N},$ $(q,a)=1,$ one has
 \begin{align*}
\pi_{c}(x,q,a):=\sum_{\substack{1\leq n\leq x\\ [n^{c}]\ \textup{is}\ \textup{prime}\\
 [n^{c}]\equiv a(\textup{mod}\ q)}}1
 =(1+o(1))\frac{N}{c\varphi(q)\log N},
 \end{align*}
 where $\varphi(n)$ is the Euler function and the implied constant may depend on $c$.
 In order to give some results related to Carmichael numbers, in \cite{BBB}, it is proved an asymptotic formula with an  explicit error term.
Let $a$ and $q$ be coprime integers, $q\geq1.$ For fixed $1<c<18/17$
 and
$\gamma=1/c,$ in \cite{BBB}, it is proved that
\begin{align*}
 \pi_{c}(x;q,a)&=\gamma x^{\gamma-1}\pi(x;q,a)-\gamma(\gamma-1)
 \int_{2}^{x}u^{\gamma-2}\pi(u;q,a)du
\\&+O\left(x^{17/39+7\gamma/13+
\varepsilon}\right),
 \end{align*}
 where
 \[
 \pi(x,q,a):=\#\{1\leq p\leq x,\ p\  \textup{is}\ \textup{prime}, \ p\equiv a(\textup{mod}\ q)\}
 \]
and the implied constant depends only on $c$ and $\varepsilon.$  This explicit error term was improved in \cite{Guo,GLZ}.  For fixed $1<c<12/11$
 and
$\gamma=1/c,$ in \cite{GLZ}, by using Heath-Brown's identity, it is proved that
\begin{align*}
 \pi_{c}(x;q,a)&=\gamma x^{\gamma-1}\pi(x;q,a)-\gamma(\gamma-1)
 \int_{2}^{x}u^{\gamma-2}\pi(u;q,a)du
\\&+O\left(x^{11/26+7\gamma/13+
\varepsilon}\right),
 \end{align*}
 where the implied constant depends only on $c$ and $\varepsilon.$
In this paper, we can give the following explicit form, which can be used to give  improved results  related to Carmichael numbers.
\begin{Theorem}\label{zpc}
Let $a$ and $q$ be coprime integers, $q\geq1.$ For fixed $1<c<12/11$
 and
$\gamma=1/c,$ we have
\begin{align*}
 \pi_{c}(x;q,a)&=\gamma x^{\gamma-1}\pi(x;q,a)-\gamma(\gamma-1)
 \int_{2}^{x}u^{\gamma-2}\pi(u;q,a)du
\\&+O\left(x^{11/14+\gamma/7+
\varepsilon}\right),
 \end{align*}
 where the implied constant depends only on $c$ and $\varepsilon.$
\end{Theorem}
\begin{Remark}
Obviously, for $1<c<12/11,$ we have
\[
11/14+\gamma/7<11/26+7\gamma/13.
\]
Hence, our result is better than previous results.
\end{Remark}
In 2013, in \cite{BBB}, by using Vaughan's identity,  Baker-Banks-Br\"udern-Shparlinski-Weingartner showed that there are infinitely many Carmichael numbers composed of Piatetski-Shapiro primes with $c\in (1,1+2/145)$.  In 2015, in \cite{Guo}, the effective range of $c$ was improved by proving  that there are infinitely many Carmichael numbers composed of Piatetski-Shapiro primes with $c\in (1,1+10/561)$. In 2023, in \cite{GLZ}, by using the generalized Vaughan's identity (Heath-Brown's identity), the effective range of $c$ was improved by proving  that there are infinitely many Carmichael numbers composed of Piatetski-Shapiro primes with $c\in (1,1+7/337)$.
In this paper, by modifying the ideas of Baker-Banks-Br\"udern-Shparlinski-Weingartner \cite{BBB}, we can show the following  result.
\begin{Corollary}
There are infinitely many Carmichael numbers composed of Piatetski-Shapiro primes with $c\in (1,1+1/30.89).$
\end{Corollary}

\section{Proof of Theorem \ref{th2}}
We will start  the proof for the Theorem \ref{th2} with some necessary lemmas.
Next lemma can be found in Theorem A.6 in \cite{GK} or Theorem 18 in \cite{Vaa}.
\begin{Lemma}\label{z1}
Suppose that $H\geq1$ and $\psi(x)=x-[x]-1/2.$ There is a function $\psi^{*}(x)$ such that
\begin{itemize}
\item
$\psi^{*}(x)=\sum_{1\leq |h|\leq H}\mathcal{R}(h)e(hx),$

\item $\mathcal{R}(h)\ll \frac1h,\ \
\left(\mathcal{R}(h)\right)'\ll \frac{1}{h^{2}},$

\item $|\psi^{*}(x)-\psi(x)|\leq
\frac{1}{2H+2}\sum_{|h|\leq H}\left(1-\frac{|h|}{H+1}\right)e(hx).$
\end{itemize}
\end{Lemma}
We also need the following result related to Dirichlet's approximation of finite type irrational number.
\begin{Lemma}\label{z2}
Let $\alpha$ be of finite type $\tau<\infty$ and let $K$ be  sufficiently  large. For an integer $w\geq1,$ there exists $a/q\in\mathbb{Q}$ with $(a,q)=1$ and $q$ satisfying $K^{1/\tau-\varepsilon}w^{-1}<q\leq K$ such that
\[
\left|\alpha w-\frac{a}{q}\right|\leq \frac{1}{qK}.
\]
\end{Lemma}
\begin{proof}
By Dirichlet approximation theorem, there is a rational number $a/q$ with $(a,q)=1$ and $q\leq K$ such that
\[
\left|\alpha w-\frac{a}{q}\right|<\frac{1}{qK}.
\]
Then we have
\[
\parallel qw\alpha\parallel\leq \frac1K.
\]
Since $\alpha$ is of type $\tau<\infty,$ for sufficiently large $K,$ we have
\[
\parallel qw\alpha\parallel\geq (qw)^{-\tau-\varepsilon}.
\]
Then we have
\[
1/K\geq \parallel qw\alpha\parallel\geq (qw)^{-\tau-\varepsilon}.
\]
This gives that
\[
q\geq K^{1/\tau-\varepsilon}w^{-1}.
\]
\end{proof}

\begin{Lemma}\label{z3}
For an arbitrary real number $\theta$ and coprime integers $c$ and $d$ with $0\leq c<d$, we have the uniform bound
\[
\sum_{\substack{1\leq  n\leq x  \\
n\equiv c(\textup{mod}d)}}\Lambda(n)e(\theta n)
\ll \left(q^{-1/2}x+q^{1/2}x^{1/2}+x^{4/5}
\right)(\log x)^{3}
\]
whenever the inequality $|\theta-a/q|\leq1/qx$ holds with some real $x>1$ and coprime integers a and $q\geq1.$
\end{Lemma}
\begin{proof}
See \cite{BP} or \cite{La}.
\end{proof}

\begin{Lemma}\label{z4}
Let $\gamma$ be an irrational number of finite type $\tau<\infty.$ For any coprime integers $c$ and $d$ with $0\leq c<d$ and any non-zero integer $k$ such that $|k|\leq x^{1/(3\tau+2)},$ we have the upper bound
\[
\sum_{\substack{1\leq n\leq x    \\
n\equiv c(\textup{mod}d)}}\Lambda(n)e(k\gamma n)
\ll x^{1-1/(3\tau+2)+\varepsilon},
\]
where the implied constant depends only on the parameters $\gamma.$
\end{Lemma}
\begin{proof}
By Lemma \ref{z2}, let
$$k=w\leq[x^{1/(3\tau+2)}]$$
and
$$K=x^{\delta(\tau,w,x)},$$
where
$$\delta(\tau,w,x)=\frac{\left(1+\frac{\log w}{\log x}\right)}{\left(1+1/\tau\right)}.$$
Hence we have
\[
K \asymp x^{3\tau/(3\tau+2)}.
\]
This means that for
\[
K^{1/\tau-\varepsilon}w^{-1}\ll q\ll K,
\]
we have
\[
x^{2/(3\tau+2)}\ll q\ll x^{3\tau/(3\tau+2)},
\]

Then by the Dirichlet approximation theorem, Lemma \ref{z3} and the fact that $\tau\geq1,$ we can get that
\[
\sum_{\substack{1\leq  n\leq x   \\
n\equiv c(\textup{mod}d)}}\Lambda(n)e(k\gamma n)
\ll x^{\varepsilon}\left(xq^{-1/2}+x^{4/5}+x^{1/2}q^{1/2}\right)
,\]
where
\[
K^{1/\tau-\varepsilon}w^{-1}\ll q\ll K.
\]
With the above choice of $K$ and $w$, we can obtain the desired conclusion immediately.
\end{proof}

Now we begin to prove Theorem \ref{th2}.
Obviously, for positive integers $m,n,$
by the relation such that
\[
m\leq n\alpha+\beta<m+1,
\]
for fixed $\beta$ of the form $\beta\neq m_{0}+k\alpha,$ $k\in\mathbb{N}$ and $1\leq m_{0}\leq N+1,$ we have
\[\sum_{1\leq  n\leq N,\ n\in\mathcal{B}_{\alpha,\beta}\atop
n\equiv c(\textup{mod}d)}\Lambda(n)
=\sum_{1\leq  m\leq N\atop m \equiv c (\textup{mod}d)}\Lambda(m)
\left(\left[\frac{m+1-\beta}{\alpha}\right]
-\left[\frac{m-\beta}{\alpha}\right]
\right).\]
For for fixed $\beta$ of the form $\beta= m_{0}+k\alpha,$ $k\in\mathbb{N}$ and $m_{0}$ being a fixed integer in the interval $[1,N+1],$ we have
\[\sum_{1\leq n\leq N,\ n\in\mathcal{B}_{\alpha,\beta}\atop
n\equiv c(\textup{mod}d)}\Lambda(n)
=\sum_{1\leq m\leq N\atop m \equiv c (\textup{mod}d)}\Lambda(m)
\left(\left[\frac{m+1-\beta}{\alpha}\right]
-\left[\frac{m-\beta}{\alpha}\right]
\right)+O(\log N).\]
Hence we have
\begin{align*}
\sum_{1\leq n\leq N,\ n\in\mathcal{B}_{\alpha,\beta}\atop
n\equiv c(\textup{mod}d)}\Lambda(n)
&=\sum_{1\leq m\leq N\atop m \equiv c (\textup{mod}d)}\Lambda(m)
\left(\left[\frac{m+1-\beta}{\alpha}\right]
-\left[\frac{m-\beta}{\alpha}\right]
\right)\\
&=\alpha^{-1}\sum_{1\leq m\leq N\atop m \equiv c (\textup{mod}d)}\Lambda(m)\\
&+\sum_{1\leq m\leq N\atop m \equiv c (\textup{mod}d)}\Lambda(m)
\left(\psi\left(\frac {m-\beta}{\alpha}\right)-\psi\left(\frac {m-\beta+1}{\alpha}\right)\right),
\end{align*}
where
$\psi(x)=x-[x]-1/2.$
Let
\[
S_{\mathcal{D}}=\sum_{1\leq  m\leq N\atop m \equiv c (\textup{mod}d)}\Lambda(m)
\psi\left(\frac {m-\beta+\mathcal{D}}{\alpha}\right),\ \mathcal{D}=0,1.
\]
By Lemma \ref{z1}, we have
\begin{align*}
\psi(t)&\leq \psi^{*}(t)+1/2(H+1)
\\&+1/2(H+1)\sum_{1\leq |h|\leq H}\left(1-|h|/(H+1)\right)e(ht)\\
&\leq 1/2(H+1)
\\&+\sum_{1\leq |h|\leq H}
\left(\mathcal{R}(h)+\frac{1}{2H+2}\left(
1-\frac{|h|}{H+1}\right)\right)e(ht)
\end{align*}
and
\begin{align*}
\psi(t)&\geq \psi^{*}(t)-1/2(H+1)
\\&-1/2(H+1)\sum_{1\leq |h|\leq H}\left(1-|h|/(H+1)\right)e(ht)\\
&\geq -1/2(H+1)
\\&+\sum_{1\leq |h|\leq H}
\left(\mathcal{R}(h)-\frac{1}{2H+2}\left(
1-\frac{|h|}{H+1}\right)\right)e(ht),
\end{align*}
where $\mathcal{R}(h)$ is defined by Lemma \ref{z1}.
Hence taking $H=N^{1/(3\tau+2)}$ in Lemma \ref{z1}, we have
\begin{align*}
S_{\mathcal{D}}&\ll H^{-1}\sum_{1\leq m\leq N\atop m \equiv c (\textup{mod}d)}\Lambda(m)\\
&+\sum_{1\leq|h|\leq H}|h|^{-1}\left|\sum_{1\leq m\leq N\atop m \equiv c (\textup{mod}d)}\Lambda(m)e(h(m-\beta+\mathcal{D})/\alpha)
\right|.
\end{align*}
Note that $\alpha$ and $\gamma=\alpha^{-1}$ are of the same type. This means that $\tau(\alpha)=\tau(\gamma)$ (see page 133 in \cite{BY}).
Then by the above assuming and Lemma \ref{z4} (exponential sums), for $0<h\leq H=N^{1/(3\tau+2)},$ we have
\begin{align*}
S_{\mathcal{D}}&\ll N^{1-1/(3\tau+2)+\varepsilon}.
\end{align*}
Hence we have
\begin{align*}
\sum_{1\leq n\leq N,\ n\in\mathcal{B}_{\alpha,\beta}\atop
n\equiv c(\textup{mod}d)}\Lambda(n)
=\alpha^{-1}\sum_{1\leq m\leq N\atop m \equiv c (\textup{mod}d)}\Lambda(m)
+O(N^{1-1/(3\tau+2)+\varepsilon}).
\end{align*}
This completes the proof of Theorem \ref{th2}.
\vspace{0.5cm}

\vspace{0.5cm}

\section{Proof of Theorem \ref{zpc}}
In order to give the proof of Theorem \ref{zpc}, we need to introduce some lemmas.
\begin{Lemma}\label{1c}
Suppose that $1<c<2$ and $\gamma=1/c.$ Let $z_{1},z_{2},\cdots$ be complex numbers
such that $z_{k}\ll k^{\varepsilon}$. Then we have
\[
\sum_{1\leq k \leq K,\ k=[n^{2}]}z_{k}
=\gamma\sum_{1\leq k\leq K}z_{k}k^{\gamma-1}+
z_{k}\left(\psi(-(k+1)^{\gamma}
-\psi(-k^{\gamma}))\right)+O(1).
\]
\end{Lemma}
\begin{proof}
See Lemma 2 in \cite{BBB}.
\end{proof}
\begin{Lemma}\label{2t}
Let $f$ be two times continuously differentiable on a subinterval
$\mathcal{I}$ of $(N, 2N].$
Suppose that for some $\lambda> 0,$ the inequalities
\[\lambda\ll f''(t)\ll \lambda\ \ (t\in \mathcal{I})
\]
hold, where the implied constants are independent of $f$ and $\lambda$. Then
\[
\sum_{n\in \mathcal{I}}e(f(n))\ll
N\lambda^{1/2}+
\lambda^{-1/2}.
\]
\end{Lemma}
\begin{proof}
See Theorem   2.2 of Graham and Kolesnik \cite{GK}.
\end{proof}
\begin{Lemma}\label{h}
Let
\[
L(Z)=\sum_{i=1}^{u}A_{i}Z^{a_{i}}
+\sum_{j=1}^{v}B_{j}Z^{-b_{j}},
\]
where $A_{i},a_{i},B_{j},b_{j}$ are positive. Let $0\leq Z_{1}\leq Z_{2} .$ Then there is
some $Z \in(Z_{1}, Z_{2}]$ with
\[
L(Z)\ll \sum_{i=1}^{u}\sum_{j=1}^{v}
\left(A_{i}^{b_{j}}B_{j}^{a_{i}}\right)
^{1/(a_{i}+b_{j})}
+\sum_{i=1}^{u}A_{i}Z_{1}^{a_{i}}
+\sum_{j=1}^{v}B_{j}Z_{2}^{-b_{j}},
\]
where the implied constant depends only on $u$ and $v.$
\end{Lemma}
\begin{proof}
See  Lemma 2.4 of \cite{GK}
\end{proof}

\begin{Lemma}\label{iq}
Let $1<Q\leq L.$ If $f$ is a function of the form $f(n)=e(g(n)),$
then
\[
\left|{\sum_{k\sim K}\sum_{l\sim L}} a_{k}b_{l} e\left(kl\right)\right|^{2}
\ll X^{2}Q^{-1}+XQ^{-1}\sum_{0<|q|< Q}\sum_{l\sim L}|S(q,l)|,\]
where
\[
S(q,l)=\sum_{k\in I(q,l)}e (g(kl)-g(k(l+q)))
\]
for a certain subinterval $I(q,l)$ of
$(X,X_{1}].$
\end{Lemma}
\begin{proof}
See Lemma 15 in \cite{GK}.
\end{proof}

\begin{Lemma}\label{t1}
Suppose $|a_{k}|\leq 1$ for all $k\sim K.$ Fix $\gamma\in(0, 1)$ and $m, h, d \in \mathbb{N.}$
Then, for any $L \gg N^{2/3},$ the Type $I$ sum
\[
S_{I}:={\sum_{k\sim K}\sum_{l\sim L}} a_{k}  e\left(hk^{\gamma}l^{\gamma}
+\frac{kl}{d}\right)
\]
satisfies the bound
\[S_{I}\ll m^{1/2}N^{1/3+\gamma/2}+ m^{-1/2}N^{1-\gamma/2}.\]
\end{Lemma}
\begin{proof}
See Lemma 11 in \cite{BBB}.
\end{proof}
By Lemma \ref{1c}, we have
\[
\pi_{c}(x,q,a):=S_{1}+S_{2}+O(1),
\]
where
\[
S_{1}=\gamma \sum_{\substack{1\leq p\leq x\\ p\equiv a(\mod q)}}p^{\gamma-1},
\]
\[
S_{2}=\gamma \sum_{\substack{1\leq p\leq x\\ p\equiv a(\mod q)}}\left(\psi(-(p+1)^{\gamma})
-\psi(-p^{\gamma})\right).
\]
By using partial summation, we have
\begin{align}\label{l1}
S_{1}=\gamma x^{\gamma-1}\pi(x,q,a)
-\gamma(\gamma-1)\int_{2}^{x}u^{\gamma-2}
\pi(u,q,a) du
\end{align}
and
\begin{align}\label{l2}
S_{2}\ll S+x^{1/2},
\end{align}
where
\begin{align}\label{l3}
S:=\sum_{\substack{N<n\leq 2N\\ [n^{c}]\equiv a(\mod q) }}\Lambda(n)\left(\psi(-(n+1)^{\gamma})
-\psi(-n^{\gamma})\right).
\end{align}
Arguing similar as in  \cite{BBB}, for any real number  we derive the
uniform bound
\begin{align}\label{l4}
S\ll N^{\gamma-1}\max_{N_{1}\sim N}
\sum_{1\leq h\leq H}\left(
\sum_{\substack{N<n\leq 2N_{1}\\ [n^{c}]\equiv a(\textup{mod}\ q) }}\Lambda(n)e(hn^{\gamma})\right)
+NH^{-1}+N^{\gamma/2}H^{1/2}.
\end{align}
To bound the inner sum, we note that
\begin{align}\label{l5}
\sum_{\substack{N<n\leq 2N\\ [n^{c}]\equiv a(\textup{mod}\ q) }}\Lambda(n)e(hn^{\gamma})
=\frac1q\sum_{1\leq t\leq q}\sum_{N<n\leq 2N}\Lambda(n)e\left(hn^{\gamma}
+\frac{(n-a)t}{q}\right).
\end{align}
Hence it suffices to give a bound on exponential sums of the form
\[
\sum_{N<n\leq 2N}\Lambda(n)e\left(hn^{\gamma}
+\frac{(n-a)t}{q}\right).
\]
In order to give the estimation for the above exponential sum, we need the following lemma.
\begin{Lemma}\label{t2}
Suppose $|a_{k}|\leq 1$ and $|b_{k}|\leq 1$ for $(k,l)\sim(K, L).$ Fix $\gamma\in(0, 1)$
and $h, t, d\in\mathbb{N}.$ Then, for any $K$ in the range  $N^{1/2}\ll K\ll N^{2/3},$  and $KL \sim N,$
 the
Type $II$ sum
\[
S_{II}:={\sum_{k\sim K}\sum_{l\sim L}} a_{k}b_{l} e\left(hk^{\gamma}l^{\gamma}
+\frac{kl}{d}\right)
\]
satisfies the bound
\[
|S_{II}|  \ll m^{1/6}N^{3/4+\gamma/6}+
N^{7/8}+m^{1/4}N^{5/8+\gamma/4}+ m^{-1/4}N^{1-\gamma/4}.
\]
\end{Lemma}
\begin{proof}
Assume that $KL\asymp N.$ By Lemma \ref{iq} , for each $I(q; l)$ being a certain subinterval in the set of numbers $k\sim K,$ we have
\[
|S_{II}|^{2}\ll K^{2}L^{2}Q^{-1}+
KLQ^{-1}\sum_{l\sim L}\sum_{1\leq |q|\leq Q}|S(q,l)|,
\]
where
\[
S(q,l)=\sum_{k\in I(q,l)}e\left(F(k)\right)
\]
and
\[
F(k)=hk^{\gamma}(l^{\gamma}
-(l+q)^{\gamma})-kqh/d.
\]
Since
\[
|F''(k)| =h\gamma
(1-\gamma)k^{\gamma-2}\left((l+q)^{\gamma}
-l^{\gamma}\right)\asymp mK^{\gamma-2}
L^{\gamma-1}q
\]
It follows from Lemma \ref{2t} that
\[
S(q,l)\ll K(hK^{\gamma-2}L^{\gamma-1})^{1/2}
+ (hK^{\gamma-2}L^{\gamma-1}q)^{-1/2}
\]
Inserting this bound in   and summing over $l$ and $q$, we derive that
\begin{align*}
|S_{II}|^{2}&\ll K^{2}L^{2}Q^{-1}
+m^{1/2}K^{1+\gamma/2}
L^{3/2+\gamma/2}Q^{1/2}+
m^{-1/2}
K^{2-\gamma/2}L^{5/2-\gamma/2}Q^{-1/2}\\
&\ll N^{2}Q^{-1}+m^{1/2}K^{-1/2}N^{3/2+\gamma/2}
Q^{1/2}
+m^{-1/2}K^{-1/2}N^{5/2-\gamma/2}Q^{-1/2},
\end{align*}
where we used $KL \asymp N.$ Since the above
holds whenever $0<Q \leq L,$ an application of Lemma \ref{h} gives
\[
|S_{II}|^{2}\ll m^{1/3}K^{-1/3}N^{5/3+\gamma/3}
+K^{-1/2}N^{2}+m^{1/2}K^{-1/2}N^{3/2+\gamma/2}
+KN+ m^{-1/2}N^{2-\gamma/2}
.
\]
Finally, for $K$ in the range $N^{1/2}\ll K \ll N^{2/3}$, we arrive at the bound
\[
|S_{II}|^{2}\ll m^{1/3}N^{3/2+\gamma/3}+N^{7/4}
+m^{1/2}N^{5/4+\gamma/2}
+m^{-1/2}N^{2-\gamma/2},
\]
and the result follows.
\end{proof}
To deal with the von Mangoldt function, we also need the following Vaughan's identity (for example, see \cite{BBB} and the references therein).
\begin{Lemma}\label{pv} There are six real arithmetical functions $\alpha_{k}(n)$ verifying $\alpha_{k}(n)\ll _{\varepsilon} n^{\varepsilon}$
for
($n>1, 1\leq k\leq 6$) such that, for all $D>100$ and any arithmetical function $g,$ we have
\[\sum_{D<d\leq 2D}
\Lambda(d)g(d) = S_{1} +  S_{2} +  S_{3} +  S_{4},\]
where
\begin{align*}
&S_{1}=\sum_{m\leq D^{1/3}}\alpha_{1}(m)\sum_{D<mn\leq 2D}g(mn),\\
&S_{2}=\sum_{m\leq D^{1/3}}\alpha_{2}(m)\sum_{D<mn\leq 2D}g(mn)\log n,\\
&S_{3}=\mathop{\sum\sum}_{\substack{D^{1/3}<m,n\leq D^{2/3}\\ D<mn\leq 2D}}\alpha_{3}(m)\alpha_{4}(n)g(mn),
\\&S_{4}=\mathop{\sum\sum}_{\substack{D^{1/3}<m,n\leq D^{2/3}\\ D<mn\leq 2D}}\alpha_{5}(m)\alpha_{6}(n)g(mn).
\end{align*}
The sums $S_{1}$ and $S_2$ are called as type $I,$ $S_3$ and $S_4$ are called as type $II.$
\end{Lemma}
Then by (\ref{l1})-(\ref{l5}), Lemma \ref{t1}, Lemma \ref{t2}, Lemma \ref{pv} and choosing
$$H=N^{3/14-\gamma/7},$$
for $1<c<12/11,$ we have
\begin{align*}
 \pi_{c}(x;q,a)&=\gamma x^{\gamma-1}\pi(x;q,a)-\gamma(\gamma-1)
 \int_{2}^{x}u^{\gamma-2}\pi(u;q,a)du
\\&+O\left(x^{11/14+\gamma/7+
\varepsilon}\right).
 \end{align*}

\bigskip
$\mathbf{Acknowledgements}$
I am deeply grateful to the referee(s) for carefully reading the manuscript and making useful suggestions.

%This work was supported by National Natural Science Foundation of China (Grant No. 11871307).

\address{Wei Zhang\\ School of Mathematics and Statistics\\
               Henan University\\
               Kaifeng  475004, Henan\\
               China}
\email{zhangweimath@126.com}

\end{document}